\documentclass[10pt,oneside,leqno]{amsart}
\usepackage{amsxtra}
\usepackage{amsopn}
\usepackage{amsmath,amsthm,amssymb}
\usepackage{amscd}
\usepackage{amsfonts}
\usepackage{latexsym}
\usepackage{verbatim}
\usepackage{pb-diagram}
\newcommand{\n}{\noindent}
\newcommand{\lr}{\longrightarrow}

\theoremstyle{plain}
\newtheorem{theorem}{Theorem}[section]
\newtheorem*{theorem*}{Theorem}
\newtheorem{definition}[theorem]{Definition}
\newtheorem{lemma}[theorem]{Lemma}
\newtheorem{prop}[theorem]{Proposition}
\newtheorem{cor}[theorem]{Corollary}
\newtheorem{rem}[theorem]{Remark}
\newtheorem{ex}[theorem]{Example}
\newtheorem*{mt*}{Main Theorem}
\newcommand\C{{\mathbb C}}

\newcommand\R{{\mathbb R}}

\begin{document}
\title[On some cohomological properties of almost complex manifolds]{On some cohomological properties of
almost complex manifolds}
\author{Anna Fino and Adriano Tomassini}
\date{\today}
\address{Anna Fino, Dipartimento di Matematica \\ Universit\`a di Torino\\
Via Carlo Alberto 10 \\
10123 Torino\\ Italy} \email{annamaria.fino@unito.it}
\address{Adriano Tomassini, Dipartimento di Matematica\\ Universit\`a di Parma\\ Viale G.P. Usberti 53/A\\
43100 Parma\\ Italy} \email{adriano.tomassini@unipr.it}
\thanks{This work was supported by the Projects MIUR ``Riemannian Metrics and Differentiable Manifolds'',
``Geometric Properties of Real and Complex Manifolds'' and by GNSAGA
of INdAM}
\begin{abstract} We study a special type of almost complex structures,
called {\em pure and full} and introduced by T.J. Li and W. Zhang in \cite{LIZhang}, in
relation to symplectic structures and Hard Lefschetz condition. We
provide sufficient conditions to the existence of the above type of
almost complex structures on compact quotients of Lie groups by discrete
subgroups. We obtain families of pure and full almost complex structures on compact
nilmanifolds and solvmanifolds. Some
of these families are parametrized by real $2$-forms which are anti-invariant with
respect to the almost complex structures.
\end{abstract}
\maketitle
\section{Introduction}

Let $M$ be a compact oriented manifold of dimension $2n$. A
symplectic form $\omega$ compatible with the orientation is a
closed $2$-form $\omega$ such that the $2n$-form $\omega^n$ is a
volume form compatible with the orientation. An almost complex
structure $J$ on a symplectic manifold $(M, \omega)$ is said to be
{\em tamed} by $\omega$ if $\omega_x (u, J u) >0$, for every $x \in
M$ and every tangent vector $u \neq 0 \in T_x M$. $J$ is called {\em
calibrated} by $\omega$ or, equivalently, $\omega$ is said to be {\em compatible} with $J$ if in
addition $\omega_x (Ju, Jv) = \omega_x (u, v)$, for any pair of
tangent vectors $u$ and $v$. In this case the pair $(\omega, J)$
is an {\em almost-K\"ahler} structure or, equivalently, $J$
is said to be {\em almost-K\"ahler}.

Let ${\mathcal C}(M)$ be the {\em symplectic cone} of $M$, i.e. the image
of the space of symplectic forms on $M$ compatible with the orientation
under the projection to the de Rham cohomology $H^2(M, \R)$. In
\cite{LIZhang} T. J. Li and W. Zhang have studied the following
subcones of ${\mathcal C} (M)$: the {\em $J$-tamed symplectic cone}, i.e
$$
{\mathcal K}_J^t (M) = \left\{ [\omega] \in H^2 (M, \R) \,\, \vert \,\, \omega \,
{\mbox {is tamed by}} \, J \right\}
$$
and the {\em $J$-compatible symplectic cone}
$$
{\mathcal K}_J^c (M) = \left\{ [\omega] \in H^2 (M, \R) \, \vert \, \omega \,
{\mbox {is compatible with}} \, J \right\}.
$$
An almost complex structure $J$ is {\em integrable} if its Nijenhuis
tensor vanishes. For almost-K\"ahler manifolds $(M, J, \omega)$,
one has that the cone ${\mathcal K}_J^c (M)$ is not empty and, when
$J$ is integrable, ${\mathcal K}_J^c(M) $ is equal to the usual
K\"ahler cone.

In \cite{LIZhang} it was studied the relation between the $J$-tamed
symplectic cone and the $J$-compatible symplectic cone in the case
of an integrable almost complex structure $J$, showing that if $
{\mathcal K}_J^c (M) $ is non-empty, then one has the splitting
$$
{\mathcal K}_J^t (M) = {\mathcal K}_J^c (M) + \left[ (H_{\overline
\partial}^{2,0} (M) \oplus H_{\overline \partial}^{0,2}
(M)) \cap H^2(M, \R) \right],$$ where $H_{\overline \partial}^{p,q} (M)$ denotes the
$(p,q)$-Dolbeault cohomology of the complex manifold $(M, J)$. In order to generalize the
previous result to the case of non-integrable almost complex
structures, they study differential forms and currents on almost
complex manifolds. The complex of currents
on an almost complex manifold $(M, J)$ is the dual to the complex
of differential forms and vice versa. On an almost complex manifold
$(M,J)$, the space of real $k$-currents ${\mathcal E}_k (M)_{\R}$
has a type decomposition:
$$
{\mathcal E}_k (M)_{\R} = \bigoplus_{p + q =k} {\mathcal E}_{p,q}^J (M)_{\R}\,.
$$ Denote by ${\mathcal Z}_{1,1}^J$
and ${\mathcal B}_{1,1}^J$ respectively the
space of real closed bidimension $(1,1)$ currents and the one of real boundary bidimension $(1,1)$
currents.
 Consider the space ${\mathcal
Z}_{(2,0),(0,2)}^J$ (respectively ${\mathcal B}_{(2,0),(0,2)}^J$)
of closed (resp. boundary) real $2$-currents which are sums of
currents of bidimension $(2,0)$ and $(0,2)$. Then, by using the
results of \cite{HL}, in \cite{LIZhang} the notions of pure and
full almost complex structure have been introduced. More
precisely, an almost complex structure $J$ on $(M, J)$ is {\em
pure} if $$\frac {{\mathcal Z}_{1,1} ^J }{{\mathcal B}_{1,1} ^J}
\cap \left( \frac{{\mathcal Z}_{(2,0),(0,2)}^J } {{\mathcal
B}_{(2,0),(0,2)}^J} \right) =0$$ and it is {\em full} if
$$
\frac{{\mathcal Z}_{2}} {{\mathcal B}_{2}}= \frac {{\mathcal
Z}_{1,1} ^J }{{\mathcal B}_{1,1} ^J} + \left( \frac{{\mathcal
Z}_{(2,0),(0,2)}^J } {{\mathcal
B}_{(2,0),(0,2)}^J} \right),
$$
where ${\mathcal Z}_{2}$ (resp. ${\mathcal B}_{2}$) are the space of real closed (boundary) $2$-currents.

One can give similar notions of ${\mathcal C}^{\infty}$ pure and
full almost complex structures by considering differential forms
instead of currents, i.e. $J$ is ${\mathcal C}^{\infty}$ pure and
full if and only if
$$
H^2 (M, \R) = H^{1,1}_J (M)_{\R} \oplus H^{(2,0),(0,2)}_J (M)_{\R},
$$
where
$$
H^{1,1}_J (M)_{\R} = \{ [\alpha] \, \vert \, \alpha \in {\mathcal Z}^{1,1}_J \}, \quad
H^{(2,0),(0,2)}_J (M)_{\R} = \{ [\alpha] \, \vert \, \alpha \in {\mathcal Z}^{(2,0),(0,2)}_J \}
$$
and ${\mathcal Z}^{1,1}_J$, ${\mathcal Z}^{(2,0),(0,2)}_J$ are defined in a similar way as before, i.e. are
respectively the closed $J$-invariant and the $J$-anti-invariant forms.
In general, there is no relation between the two notions. If $J$ a smooth closed almost
complex structure, i.e.
it is such that the image of the space ${\mathcal B}_2$ of real boundary $2$-currents under
the projection $\pi_{1,1}^J$
is a (weakly) closed subspace of
${\mathcal E}_{1,1}^J (M)_{\R}$, then some relations between the spaces $H^{1,1}_J (M)_{\R}$ and $ \frac {{\mathcal
Z}_{1,1} ^J }{{\mathcal B}_{1,1} ^J}$ are found in \cite{LIZhang}. Moreover, on a compact manifold of real dimension
$4$ any almost complex structure is
${\mathcal C}^{\infty}$
pure and
full by \cite[Theorem 2.3]{DLZ}.

In Section \ref{calibrated} we review some known facts about
calibrated almost complex structures and some properties of almost-K\"ahler manifolds. In
Section \ref{maintheorems} we study
${\mathcal C}^{\infty}$ full and pure almost complex structures on
 compact quotients $M = \Gamma \backslash G$, where $\Gamma$ is a
uniform discrete subgroup of a Lie group $G$, such that the de
Rham cohomology $H^2(M, \R)$ of $M$ is isomorphic to the
Chevalley-Eilenberg cohomology $H^2 ({\mathfrak g} )$ of the Lie
algebra ${\mathfrak g}$ of $G$. In this case we show that $ H^{1,1}_J (M)_{\R}$ and $H^{(2,0),(0,2)}_J
(M)_{\R}$ can be determined by using invariant forms (Theorem \ref{relwithinvforms}).

If $\omega$ is a non-degenerate $2$-form on a $2n$-dimensional compact manifold $M$, then we prove that
 a ${\mathcal C}^{\infty}$
pure and full almost complex structure $J$ calibrated by $\omega$ is pure. Moreover,
if, in addition, either $n = 2$ or if any
cohomology class in $H^{1,1}_J (M)_{\R}$ ($H^{(2,0),(0,2)}_J (M)_{\R}$ respectively) has a
harmonic representative in ${\mathcal Z}^{1,1}_J$
(${\mathcal Z}^{(2,0),(0,2)}_J$ respectively) with respect to the metric induced by $\omega$ and $J$, then
$J$ is pure and full (Theorem \ref{harmonic}).
We give examples of compact non-K\"ahler solvmanifolds, i.e. compact quotients of solvable Lie groups
by discrete
subgroups, satisfying the previous conditions.

In Section \ref{hardLefschetz} we prove that on a compact symplectic manifold which
satisfies the Hard
Lefschetz condition a ${\mathcal C}^{\infty}$ pure and full almost complex structure is pure and full.

An integrable almost complex structure $J$ is closed but, in general, it is not necessarily (${\mathcal C}^{\infty}$)
pure and full. We show that
a complex parallelizable manifold $(\Gamma/G, J)$ for which $H^2 (M, \R)$ is isomorphic
to the Chevalley-Eilenberg
cohomology $H^2 ({\mathfrak g})$, then $J$ is ${\mathcal C}^{\infty}$ full and it is pure (Theorem \ref{complexparal}). We
provide an example of
complex parallelizable manifold, endowed with a pure and full complex structure, namely the {\em Nakamura manifold}, for which
the de Rham cohomology is not isomorphic to the
Chevalley cohomology of the corresponding solvable Lie algebra (see \cite{Nakamura} and \cite{debaT}).

By using \cite[Proposition 1.16]{Audin} in section \ref{familyLafontaine} we provide a
family of pure and full almost
complex structures on a $3$-step nilmanifold of real dimension $4$ and on two $2$-step
solvmanifolds of real dimension $4$ and $6$.

A {\em $J$-holomorphic map} into an almost complex manifold $(M, J)$ is
a map $f: (\Sigma, j) \to (M, J)$ from a compact Riemann surface $(\Sigma, j)$ whose
differential is complex linear. Gromov-Witten invariants of a
symplectic manifold are counts of holomorphic curves with respect
to an almost complex structure compatible with the symplectic
structure (see e.g. \cite{LT}). In \cite{Lee} it was shown that, when $M$ is
$4$-dimensional, there is a natural infinite-dimensional family of
almost complex structures $J_{\alpha}$ parametrized by the
$J$-anti invariant real $2$-forms on $M$. \newline
In the last section we
construct an explicit family $J_{\alpha}$ of pure and full almost
complex structures on a $4$-dimensional nilmanifold and on a $4$-dimensional solvmanifold.

\smallskip

\noindent {\em{Acknowledgements.}} We would like to thank Tian-Jun Li and Weiyi Zhang for useful comments, remarks and
for pointing us the reference \cite{DLZ} and Daniele Angella for useful comments.\newline
We also would like to thank the referee for valuable
remarks and suggestions for a better presentation of the results.
\smallskip

\section{Calibrated almost complex structures}\label{calibrated}
Let $(V,\,\omega)$\, be a $2n$-dimensional symplectic real vector space. We recall the following

\begin{definition}
A (linear) complex structure $J$ on $(V, \omega)$\, is said to be $\omega$-{\em calibrated}\, if:
\begin{itemize}
\item[1.] $\omega(Jv,\,Jw)=\omega(v,\,w)\,,$\vskip.1truecm
\item[2.] $g_{J}(v,\,w):=\omega(v,\,Jw)$\, is positive definite,
\end{itemize}
 for every $v,\,w \in V$.
\end{definition}

Denote by ${\frak{C}}_{\omega}(V)$\, the set of $\omega$-calibrated linear complex structures on $V$.
Then, it is well known (see e.g. \cite{Audin}) that:
$${\mathfrak{C}}_{\omega}(V)\simeq{\frak{C}}(n):=Sp(n,\,{\Bbb{R}})/U(n)$$
and
$$
{\mathcal{C}}(n):=\{X\in M_{2n,2n}(\R)\,\,\,|\,\,\,X={}^{t}X,\,\,\,XJ_{n}+J_{n}X=0\}\,,
$$
where $M_{2n,2n}(\R)$ denotes the set of real matrices of order $2n$ and $J_{n}$ is the standard
complex structure
on $\R^{2n}$. Therefore,
${\frak{C}}_{\omega}(V)$\, is homeomorphic to an $(n^{2}+n)$-dimensional cell.

\begin{comment}
\n Given $J\in{\frak{C}}_{\kappa}(V)$\,, $\wedge^{*}(V):=\wedge^{*}V^{*}\otimes{\Bbb{C}}$\, is ${\Bbb{Z}}^{+}$-bigraded
with respect to the types as
$$\wedge^{*}(V)=\bigoplus_{r=0}^{2n}\wedge^{r}(V)=\bigoplus_{r=0}^{2n}\bigoplus_{p+q=r}\wedge_{J}^{p,q}(V)$$
\n and also ${\Bbb{Z}}_{2}$-graded as
\begin{equation}\label{E:z2}
\wedge^{*}(V)=\wedge^{*}_{+}(V)\oplus\wedge^{*}_{-}(V)
\end{equation}
\n according to the parity of $q$\, with respect to the $J-(p,\,q)$-bigraduation;

\n For $J\in{\frak{C}}_{\kappa}(V)$\,, let
$*\,:\, \wedge^{p,q}_{J}(V)\,\lr\,\wedge^{n-q,n-p}_{J}(V)$\,, defined by:
$$\alpha\wedge\overline{*\beta}=g_{J}(\alpha,\bar{\beta})\frac{\kappa^{n}}{n!}$$

\n we have the following, easy to prove two lemmata:
\begin{lemma}\label{L:lf}
\n we have:
\begin{itemize}
\item[1.] $*^{2}=(-1)^{p+q}I$\, on $\wedge_{J}^{p,q}(V)$
\item[2.] $*\bar{\alpha}=\overline{*\alpha}$\,, i.e. $*$\, is a ${\Bbb{C}}$-linear real operator
\item[3.] $*J=J*=\bigstar$
\item[4.] $\Lambda =-*^{-1}L*$
\end{itemize}
\end{lemma}
\end{comment}
For any integer $k\,,\,0\leq k\leq n$, let $L^{k}\,:\,\wedge^{n-k}V^{*}\,\lr\,\wedge^{n+k}V^{*}$ be
the linear map defined by
$$
L^k(\alpha) =\alpha\wedge\omega^k\,.
$$
We will need the following (see e.g. \cite[Corollary 2.7]{Yan})
\begin{lemma}\label{L:iso}
For any integer $k\,,\,0\leq k\leq n$\,\,,\,\, the map $L^{k}\,:\,\wedge^{n-k}V^{*}\,\lr\,
\wedge^{n+k}V^{*}$\,
is an isomorphism.
\end{lemma}

\vspace{0.2cm}
Let $(M, \omega)$ be a symplectic manifold.
\begin{definition}
An almost complex structure
$J$ is said to be {\em tamed} by $\omega$ if $\omega_x (u, J u) >0$,
for any $x \in M$ and any tangent vector $u \neq 0 \in T_x M$. $J$ is called
{\em calibrated} by $\omega$ if in addition $\omega_x (Ju, Jv) = \omega_x (u, v)$, for any pair
of tangent vectors $u$ and $v$.
\end{definition}

For a fixed non-degenerate closed $2$-form $\omega$ on $\R^{2n}= \C^n$, denote by
${\mathcal J}_c(\omega)$ (resp. ${\mathcal J}_t (\omega)$) the set of almost-complex structures
calibrated (respectively tamed) by $\omega$.

By \cite[Proposition 1.16]{Audin}, if on $\C^n$ one considers the canonical standard
symplectic structure $(J_0, \omega)$, then the map
$$
J \mapsto (J + J_0)^{-1} \circ (J - J_0)
$$
is a diffeomorphism from ${\mathcal J}_t (\omega)$ (resp. ${\mathcal J}_c(\omega)$) onto the
open unit ball in the vector space of matrices (resp. symmetric matrices) $L$ such that $J_0 L = - L J_0$.

Indeed, if $L$ is matrix such that $\vert \vert L \vert \vert < 1$, then the endomorphism
$$
J_0 \circ (I+ L) \circ (I - L)^{-1}\,,
$$
where $I$ is the identity matrix, is an almost complex
structure if and only if $J_0 L = - LJ_0$ and it is tamed by $\omega$.
Moreover, the almost complex structure
$$
J_0 \circ (I + L) \circ (I - L)^{-1}
$$
is calibrated if and only if $L$ is symmetric.

Then, if $J_0$ is a almost complex structure calibrated by $\omega$ and $L$ is a symmetric matrix
such that $\vert \vert L \vert \vert < 1$, $J_0 L = - L J_0$, then the new endomorphism
$$
(I + L) \circ J_0 \circ (I + L)^{-1}
$$
is still an almost complex structure calibrated by $\omega$.

Let again $(M,\,\omega)$\, be a $2n$-dimensional symplectic manifold and let $Sp_{\omega}(M)$\, be
the principal $Sp(n,\,{\Bbb{R}})$-bundle of symplectic frames on $M$\,; then $Sp_{\omega}(M)/U(n)$\, is
the bundle of $\omega$-calibrated almost complex structures on $M$. Since $Sp_{\omega}(M)$ has
contractible fiber, there exist
 many global sections. Denote by ${\frak{C}}_{\omega}(M)$\, the space of such sections.

Given an $\omega$-calibrated almost complex structure $J\in {\frak{C}}_{\omega}(M)$, the space $\Omega^k (M)$
of complex smooth differential
$k$-forms has a type decomposition:
$$
\Omega^k (M) = \bigoplus_{p + q = k} \Omega^{p,q}_J (M),
$$
where $\Omega^{p,q}_J (M)$ denotes the space of complex forms of type $(p, q)$ with respect to $J$.
We have that:
$$
d\,:\,\Omega_{J}^{p,q}(M)\,\lr\,\Omega_{J}^{p+2,q-1}(M)\oplus\Omega_{J}^{p+1,q}(M)
\oplus\Omega_{J}^{p,q+1}(M)\oplus
\Omega_{J}^{p-1,q+2}(M)$$
\n and so $d$\, splits accordingly as
$$
d=A_{J}+\partial_{J}+\bar{\partial}_{J}+\bar{A}_{J}
$$
\n where all the pieces are graded algebra derivations, $A_{J},\,\bar{A}_{J}$\,are $0$-order
differential operators,
and, in particular, for $\alpha\in\Omega^{1}(M)$\, we have
$$
(A_{J}(\alpha)+\bar{A}_{J}(\alpha))(X,\,Y)=\frac{1}{4}\alpha(N_{J}(X,\,Y))\,,
$$
\n $N_{J}$\, being the Nijenhuis tensor of $J$\,. From $d^{2}=0$\, we obtain
$$
%\begin{equation}\label{E:cxi}
\left\{
\begin{aligned}
& A_{J}^{2}=0\,,\\
& A_{J}\partial_{J}+\partial_{J}A_{J}=0\,,\\
& \partial_{J}^{2}+A_{J}\bar{\partial}_{J}+\bar{\partial}_{J}A_{J}=0\,,\\
& \partial_{J}\bar{\partial}_{J}+\bar{\partial}_{J}\partial_{J}+A_{J}\bar{A}_{J}+\bar{A}_{J}A_{J}=0\,,\\
& \bar{\partial}_{J}^{2}+\bar{A}_{J}\partial_{J}+\partial_{J}\bar{A}_{J}=0\,,\\
& \bar{A}_{J}\bar{\partial}_{J}+\bar{\partial}_{J}\bar{A}_{J}=0\,,\\
& \bar{A}_{J}^{2}=0\,.
\end{aligned}
\right.
%\end{equation}
$$
If $N_J =0$, then $J$ is called {\em integrable} and $(\omega, J)$ defines a K\"ahler structure. In general

\begin{definition}
An {\em almost-K\"ahler structure} on $2n$-dimensional manifold $M$ is a pair $(\omega, J)$ where $\omega$ is a
symplectic form and $J$ is an almost complex structure calibrated by $\omega$.
\end{definition}
If $(M,\omega,J)$ is an almost-K\"ahler manifold, then
$$
g(X,Y ) = \omega(X,JY )
$$
is a $J$-Hermitian metric, i.e. $g(JX, JY) = g (X, Y)$, for any $X. Y$.
\section{Pure and full almost complex structures and differential forms}\label{maintheorems}
Let $J$ be a smooth almost complex structure on a compact $2n$-dimensional manifold $M$.

The space $\Omega^k (M)_{\R}$ of real smooth differential $k$-forms has a type decomposition:
$$
\Omega^k (M)_{\R} = \bigoplus_{p + q = k} \Omega^{p,q}_J (M)_{\R},
$$
where
$$
\Omega^{p,q}_J (M)_{\R} = \left\{ \alpha \in \Omega^{p,q}_J (M) \oplus \Omega^{q,p}_J (M)
\, \vert \, \alpha = \overline \alpha \right\}\,.
$$

For a finite set $S$ of pairs of integers, let
$$
{\mathcal Z}^{S} _J = \displaystyle{\bigoplus_{(p, q) \in S}} {\mathcal Z}^{p,q}_J, \quad
{\mathcal B}^{S} _J = \bigoplus_{(p, q) \in S} {\mathcal B}^{p,q}_J,
$$
where the spaces ${\mathcal Z}^{p,q}_J$ and ${\mathcal B}^{p,q}_J$
are respectively the space of real $d$-closed $(p,q)$-forms and
those one of $d$-exacts $(p,q)$-forms. There is a natural map
$$
\rho_S: {\mathcal Z}^{S} _J/{\mathcal B}^{S} _J \to {\mathcal Z}^{S} _J/{\mathcal B}
$$
where $\mathcal B$ is the space of $d$-exact forms. As in \cite{LIZhang} we will write
$\rho_S ({\mathcal Z}^{S} _J/{\mathcal B}^{S} _J)$ simply as ${\mathcal Z}^{S} _J/{\mathcal B}^{S} _J$
and we may define the cohomology spaces
$$
H^S_J (M)_{\R}= \left\{ [\alpha ] \,\, \vert \, \,\alpha \in {\mathcal Z}^{S} _J \right\} =
\frac {{\mathcal Z}^{S} _J } {\mathcal B}.
$$
Then there is a natural inclusion
$$
H^{1,1}_J (M)_{\R} + H^{(2,0),(0,2)}_J (M)_{\R} \subseteq
H^{2} (M, \R).
$$
As in \cite[Definition 4.12]{LIZhang} we set the following
\begin{definition}
A smooth almost complex structure $J$ on $M$ is said to be ${\mathcal C}^{\infty}$ {\em pure and full} if
$$
H^{2} (M, \R) =H^{1,1}_J (M)_{\R} \oplus H^{(2,0),(0,2)}_J (M)_{\R}.
$$
\end{definition}
In particular $J$ is ${\mathcal C}^{\infty}$ pure if and only if
$$
\frac{{\mathcal Z}^{1,1}_J}{{\mathcal B}^{1,1}_J} \cap
\frac{{\mathcal Z}^{(2,0),(0,2)}_J}{{\mathcal B}^{(2,0),(0,2)}_J} =0
$$
and $J$ is ${\mathcal C}^{\infty}$ full if and only if
$$
\frac{{\mathcal Z}^2}{{\mathcal B}^2} = \frac{{\mathcal Z}^{1,1}_J}{{\mathcal B}^{1,1}_J} +
\frac{{\mathcal Z}^{(2,0),(0,2)}_J}{{\mathcal B}^{(2,0),(0,2)}_J },
$$
where ${\mathcal Z}^{2}$ and ${\mathcal B}^{2}$ denote
respectively the space of $2$-forms which are $d$-closed and
exact. Let $\pi_{1,1}: \Omega^2 (M)_{\R} \to \Omega^{1,1}_J
(M)_{\R}$ be the natural projection. If $J$ is ${\mathcal
C}^{\infty}$ pure and full, then the natural homomorphism
$$
\frac {{\mathcal Z}^{1,1}_J} {{\mathcal B}^{1,1}_J} \to
\frac{\pi_{1,1} {\mathcal Z}^{2}} {\pi_{1,1} {\mathcal B}^{2}}
$$
is an isomorphism
(see \cite[Lemma 4.9]{LIZhang}).
\begin{prop} \label{co-closed} Let $\omega$ be a symplectic form on a $2n$-dimensional compact manifold $M$. If $J$
is an almost
complex structure on $M$ calibrated by $\omega$, then $J$ is ${\mathcal C}^{\infty}$ pure. \end{prop}
\begin{proof} Let $a\in
\frac{{\mathcal Z}^{1,1}_J}{{\mathcal B}^{1,1}_J} \cap
\frac{{\mathcal Z}^{(2,0),(0,2)}_J}{{\mathcal B}^{(2,0),(0,2)}_J}$. Then $a=[\alpha]=[\beta]$, where
$\alpha\in{\mathcal Z}^{1,1}_J$, $\beta\in{\mathcal Z}^{(2,0),(0,2)}_J$ respectively. We have to show
that $a=0$. By $[\alpha]=[\beta]$, it follows that $\alpha =\beta + d\gamma$, for
$\gamma\in\Omega^1 (M)_{\R}$. \newline
If $(\, , \, )$ denotes the $L^2$-product on $\Omega^k (M)_{\R}$, then we get:
\begin{equation}\label{orthomega}
(\alpha,\omega)=(\beta + d\gamma,\omega)=(d\gamma,\omega)=(\gamma,d^*\omega)=0\,,
\end{equation}
where $*$ is the star Hodge operator with respect to the Riemannian metric associated with
$(\omega,J)$. The $2$-form $\omega$ determines an $U(n)$-equivariant map
$$
L: \Omega^{p-1,q-1}_J(M) \to \Omega^{p,q}_J (M),
$$
and the orthogonal decomposition
$$
\Omega^{p,q}_J (M) = {\Omega_0}^{p,q}_J (M) \oplus L( \Omega^{p-1,q-1}_J (M))\,,
$$
where ${\Omega_0}^{p,q}_J (M)$ denotes the space of primitive forms of type $(p,q)$, i.e. the
space of $(p, q)$-forms $\beta$ such that $\beta \wedge \omega =0$.
In particular
$$
\Omega^{2} (M)_{\R} = {\Omega_0}^{1,1}_J (M)_{\R} \oplus < \omega> \oplus \,
{\Omega}^{(2,0),(0,2)}_J (M)_{\R}.
$$
Since $\alpha$ is a real form of type $(1,1)$ and by \eqref{orthomega} it is
orthogonal to $\omega$, it follows that $\alpha$ is primitive and, consequently, if $n > 2$
$$
\alpha\wedge\omega^{n-1} =0\,.
$$
Then, by Lemma \ref{L:iso}, it follows that $\alpha =0$. Hence $a=0$.

If $n =2$ we have that the spaces $ {\Omega_0}^{1,1}_J (M)_{\R} $ and
$< \omega> \oplus {\Omega}^{(2,0),(0,2)}_J (M)_{\R}$ are respectively the spaces
of self-dual and anti-self-dual $2$-forms.

Therefore, both $\alpha$ and $\beta$ are harmonic forms. By the assumption,
$$
[\alpha-\beta]=0\,,
$$
and by the harmonicity of $\alpha$ and $\beta$, we get $\alpha=\beta$. Therefore, $\alpha =\beta =0$, i.e.
$a=0$.
\end{proof}

By the last proposition, we obtain at once that, if $(\omega, J)$
is an almost-K\"ahler structure then $J$ is ${\mathcal
C}^{\infty}$ pure. This result has been also proved in \cite{DLZ}.

By \cite[Theorem 2.3]{DLZ} on a compact manifold of real dimension $4$ any almost complex structure
is ${\mathcal C}^{\infty}$ pure and full. We will show in the next example that a compact
manifold of real dimension $6$ may admit non ${\mathcal C}^{\infty}$ pure almost structures.

\begin{ex}
{\rm Consider the $6$-dimensional nilmanifold $M$, compact quotient of the $6$-dimensional real nilpotent
Lie group with structure equations
$$
\left\{
\begin{array}{l}
d e^j = 0, \qquad j = 1, \ldots, 4,\\[5pt]
d e^5= e^{12},\\[5pt]
d e^6 =e^{13},
\end{array}
\right.
$$
by a uniform discrete subgroup. The left-invariant almost complex
structure on $M$, defined by the $(1,0)$-forms
$$
\eta^1 = e^1 + i e^2, \quad \eta^2 = e^3 + i e^4, \quad \eta^3 =
e^5 + i e^6,
 $$
is not ${\mathcal C}^{\infty}$ pure, since one has that
$$[{\mbox {Re}} (\eta^1\wedge \overline \eta^2) ]=[e^{13} + e^{24}] =
 [e^{24}]= [{\mbox {Re}} (\eta^1\wedge \eta^2)] = [e^{13} - e^{24}].
 $$
}
\end{ex}

\smallskip

If $(\omega, J)$ is a K\"ahler structure, then of course $J$ is
${\mathcal C}^{\infty}$ pure and full. Therefore, the interesting case is
to find examples of ${\mathcal C}^{\infty}$ full and pure almost complex structures, not associated
to K\"ahler structures.
\begin{theorem} \label{relwithinvforms} If $J$ is a ${\mathcal C}^{\infty}$ full and pure almost
complex structure on a
compact quotient $M = \Gamma \backslash G$, where $\Gamma$ is a uniform
discrete subgroup of a Lie group $G$, such
that the de Rham cohomology $H^2(M, \R)$ of $M$ is isomorphic to the Chevalley-Eilenberg cohomology
$H^2 ({\mathfrak g} )$
of the Lie algebra ${\mathfrak g}$ of $G$, then the following isomorphisms hold:
$$
H^{1,1}_J (M)_{\R} \cong \frac{\pi_{1,1}{\mathcal Z}^2_{inv}} {\pi_{1,1} {\mathcal B}^2} ,
\quad H^{(2,0),(0,2)}_J (M)_{\R} \cong \frac {\pi_{(2,0),(0,2)} {\mathcal Z}^2_{inv} } {\pi_{(2,0),(0,2)} {\mathcal B}^2},
$$
where ${\mathcal Z}^{2}_{inv}$ denotes
the space of closed
left-invariant $2$-forms on $M$.
\end{theorem}
\begin{proof} By the assumption on the de Rham cohomology of $M$, we have
$$
H^{2} (M, \R) = \frac {{\mathcal Z}^2} {{\mathcal B}^2} \cong
H^{2}_{inv} (M, \R) = \frac {{\mathcal Z}^2_{inv}} {{\mathcal B}^2_{inv}} \,,
$$
where $H^{2}_{inv} (M, \R)\cong H^2 ({\mathfrak g} )$ is the
cohomology of complex of left-invariant $2$-forms on $M$. Let $[\alpha] \in H^{1,1}_J (M)_{\R}$, then
$[\alpha] = [\tilde \alpha]$, with $\tilde \alpha$ a left-invariant closed $2$-form. Therefore
$$
 \alpha = \pi_{1, 1} \tilde \alpha + \pi_{1,1} d \gamma,
 $$
 with $\tilde \alpha$ a left-invariant form closed $2$-form
i.e. $[\alpha] = [\pi_{1,1} \tilde \alpha]$ in $ \pi_{1,1} {\mathcal Z}^2_{inv}/ \pi_{1,1} {\mathcal B}^2$ and similarly
for an element of $H^{(2,0), (0,2)}_J (M)_{\R}$.
Since
$J$ is ${\mathcal C}^{\infty}$ pure and full we have that
$$H^{1,1}_J (M)_{\R} = \frac {{\mathcal Z}^{1,1}_J}
{{\mathcal B}^{1,1}_J} \cong \frac {\pi_{1,1} {\mathcal Z}^2} { \pi_{1, 1}{\mathcal B}^2}, \quad H^{(2,0),(0,2)}_J (M)_{\R} =
\frac {{\mathcal Z}^{(2,0), (0,2)}_J}
{{\mathcal B}^{(2,0), (0,2)}_J} \cong \frac {\pi_{(2,0), (0,2)} {\mathcal Z}^2} { \pi_{(2,0), (0,2)}{\mathcal B}^2}.
$$
Then the theorem follows by
$$
H^{2} (M, \R) = \frac {\pi_{1,1} {\mathcal Z}^2_{inv}} {\pi_{1,1} {\mathcal B}^2} \oplus \frac {\pi_{(2,0),(0,2)}
{\mathcal Z}^2_{inv}} {\pi_{(2,0),(0,2)} {\mathcal B}^2}.
$$
\end{proof}
The previous assumption on the de Rham cohomology of $\Gamma
\backslash G$ is not so restrictive, since it is satisfied when $G$ is completely solvable, so, in particular, if $G$ is
nilpotent.

The complex of currents is dual to the complex of forms and vice versa. Since the smooth $k$-forms
can be considered as $(2n-k)$-currents, the $k$-th de Rham homology group
$H_k(M, \R)$ is isomorphic to the $(2n - k)$-th de Rham cohomology group $H^{2n- k} (M, \R)$
of $M$ (see \cite{dR}).\newline
Moreover, a $k$-current is a boundary if and only if it vanishes on the space of closed $k$-forms.

Indeed, on an almost complex manifold $(M, J)$ the space of real $k$-currents ${\mathcal E}_k (M)_{\R}$
has a decomposition:
$${\mathcal E}_k (M)_{\R} = \bigoplus_{p + q =k} {\mathcal E}_{p,q}^J (M)_{\R},$$
where ${\mathcal E}_{p,q}^J (M)_{\R}$ denotes the space of real $k$-currents of bidimension $(p,q)$.

If $S$ is a finite set of pairs of integers, then let
$$
{\mathcal Z}_{S} ^J = \bigoplus_{(p, q) \in S} {\mathcal Z}_{p,q}^J, \quad {\mathcal B}_{S} ^J =
\bigoplus_{(p, q) \in S} {\mathcal B}_{p,q}^J,
$$
where the spaces ${\mathcal Z}_{p,q}^J$ and ${\mathcal B}_{p,q}^J$ are respectively the
space of real closed bidimension $(p,q)$ currents and the one of real exact bidimension $(p,q)$
currents. Let, as in
\cite{LIZhang},
$$
H_S^J (M)_{\R}= \{ [\alpha ] \, \vert \, \alpha \in {\mathcal Z}_{S}^J \} = \frac {{\mathcal Z}_{S} ^J }
{\mathcal B},
$$
where $\mathcal B$ denotes the space of currents which are boundaries.

Denote by ${\mathcal Z}_{2}$ and ${\mathcal B}_{2}$ respectively the space
of real $2$-currents which are closed and boundaries.

We recall the following (see \cite[Definitions 4.3 and
4.4]{LIZhang})
\begin{definition} An almost complex structure $J$ is said to be {\em pure}
if
$\frac {{\mathcal Z}_{1,1} ^J }{{\mathcal B}_{1,1} ^J } \cap \frac {{\mathcal Z}_{(2,0),(0,2)} ^J }
{{\mathcal B}_{(2,0),(0,2)} ^J } =0$, or
equivalently, if and only
if $\pi_{1,1} {\mathcal B}_2 \cap {{\mathcal Z}_{1,1} ^J } = {\mathcal B}_{1,1}^J$.

$J$ is said to be {\em full} if
$$
\frac {{\mathcal Z}_{2}} {{\mathcal B}_{2}}= \frac {{\mathcal Z}_{1,1} ^J }{{\mathcal B}_{1,1} ^J }
 + \frac {{\mathcal Z}_{(2,0),(0,2)} ^J } {{\mathcal B}_{(2,0),(0,2)} ^J }.
$$
\end{definition}

Therefore, an almost complex structure $J$ is pure and full if and only if
\begin{equation} \label{fullforcurrents}
H_{2} (M, \R) =H_{1,1}^J (M)_{\R} \oplus H_{(2,0),(0,2)}^J (M)_{\R},
\end{equation}
where $H_{2} (M, \R)$ is the $2$-nd de Rham homology group. In \cite{LIZhang} a link between
the two notions of pure
and full in terms of differential
forms and currents was found.
\begin{definition} An almost complex structure $J$ is said to be ${\mathcal C}^{\infty}$ {\em closed} if the image of
the operator
$$
d_{1,1}: \Omega^{1,1}_J(M)_{\R} \to \Omega^3(M)_{\R}
$$
is closed.
\end{definition}

By \cite{LIZhang} an almost complex structure $J$ is closed if and only if $J$ is ${\mathcal C}^{\infty}$ closed.
Moreover, in \cite{LIZhang} relations between the spaces $H^{1,1}_J (M)_{\R}$ and $H_{1,1}^J (M)_{\R}$ have been determined.

If a $2$-form $\omega$ on a $2n$-dimensional manifold is not necessarily closed but it is only
non-degenerate, the manifold $(M, \omega)$ is called {\em {almost symplectic}} manifold.
We can prove the following
\begin{theorem} \label{harmonic}
Let $(M, \omega)$ be an almost symplectic $2n$-dimensional compact manifold and
$J$ be a ${\mathcal C}^{\infty}$
pure and full almost complex structure calibrated by $\omega$. Then $J$ is pure. \newline
If, in addition, either $n = 2$ or if any
cohomology class in $H^{1,1}_J (M)_{\R}$ ($H^{(2,0),(0,2)}_J (M)_{\R}$ respectively) has a
harmonic representative in ${\mathcal Z}^{1,1}_J$
(${\mathcal Z}^{(2,0),(0,2)}_J$ respectively) with respect to the metric induced by $\omega$ and $J$, then
$J$ is pure and full.
\end{theorem}
\begin{proof} We start to prove that $J$ is pure. We show that
$\pi_{1,1} {\mathcal B}_2 \cap {{\mathcal Z}_{1,1} ^J } = {\mathcal B}_{1,1}^J$. \newline
Since
$\pi_{1,1} {\mathcal B}_2 \cap {{\mathcal Z}_{1,1} ^J }\supset {\mathcal B}_{1,1}^J$, it will be
enough to prove the other inclusion.\newline
Let $T \in \pi_{1,1} {\mathcal B}_2 \cap {{\mathcal Z}_{1,1} ^J }$; then $T = \pi_{1,1} dS$, where $S$ is
a real $3$-current and $d (\pi_{1,1} d S) =0$. We have to show that $T=\pi_{1,1} d S$ is a boundary, i.e.
that it vanishes on any closed real $2$-form $\alpha$.

If $\alpha$ is exact, then one has immediately that $(\pi_{1,1} d S)(\alpha) = 0$. Suppose
that $[\alpha] \neq 0 \in H^2(M, \R)$, then since $J$ is ${\mathcal C}^{\infty}$ pure and full, we can write
$$
\alpha = \alpha_1 + \alpha_2 + d \gamma,
$$
with $\alpha_1 \in {\mathcal Z}^{1,1}_J$, $\alpha_2 \in {\mathcal Z}^{(2,0),(0,2)}_J$ and
$\gamma \in \Omega^1 (M)_{\R}$. Therefore,
$$
T(\alpha)=(\pi_{1,1} d S) (\alpha)
= (\pi_{1,1} dS) (\alpha_1 + \alpha_2) = (dS) (\alpha_1) =0\,,
$$
since $\alpha_{1}$ is closed.

If $n = 2$ in order to prove that $J$ is also full, we have to show that \eqref{fullforcurrents}
holds.\newline
Let $[T] \in H_{2} (M, \R)$; then there exists a smooth closed $2$-form $\beta$ on $M$ such that
$[T] = [\alpha]$. Since $J$ is ${\mathcal C}^{\infty}$ full, we have
that $[\alpha] = [\alpha_1] + [\alpha_2]$,
with $\alpha_1 \in {\mathcal Z}^{1,1}_J$ and $\alpha_2 \in {\mathcal Z}^{(2,0),(0,2)}_J$, then
$\alpha_1$ and $\alpha_2$ can be viewed as elements of $ {\mathcal Z}_{1,1}^J$
and ${\mathcal Z}_{(2,0),(0,2)}^J$ respectively.

If $n > 2$ and for any cohomology class in $H^{1,1}_J (M)_{\R}$ ($H^{(2,0),(0,2)}_J (M)_{\R}$
respectively) one can find a harmonic representative in ${\mathcal Z}^{1,1}_J$
(${\mathcal Z}^{(2,0),(0,2)}_J$ respectively), then in order to prove that $J$ is
full we can argue in the following way. \newline
Let $[T] \in H_{2} (M, \R)$, then there exists a smooth harmonic $(2n-2)$-form $\beta$ on
$M$ such that $[T] = [\beta]$. Then, the $2$-form $\gamma =*\beta$
is closed and defines a cohomology class $[\gamma] \in H^2(M, \R)$. By
the assumption, thus there exist real harmonic forms $\gamma_1 \in \Omega^{1,1}_J(M)_{\R}$
and $ \gamma_2 \in \Omega^{(2,0),(0,2)}_J(M)_{\R}$ such that
$$
[\gamma] = [\gamma_1] + [\gamma_2].
$$
The $(2n- 2)$-forms $\beta_1=*\gamma_1$ and $\beta_2=*\gamma_2$ then can be viewed as elements
respectively of
$ {\mathcal Z}_{1,1}^J$ and ${\mathcal Z}_{(2,0),(0,2)}^J$, i.e.
$$
[T] = [\beta_1] + [\beta_2]\,.
$$
\end{proof}
\begin{rem} {\rm {The assumption in Theorem \ref{harmonic} that any $(1,1)$ (respectively $(2,0)+(0,2)$)
de Rham class contains a harmonic representative seems quite strong. Nevertheless, we will provide several
examples of compact non-K\"ahler solvmanifolds satisfying the above assumption. In order to get the pureness of $J$,
it is enough to assume that $J$ is ${\mathcal C}^\infty$ full (see also \cite{LIZhang}). }}
\end{rem}
\begin{rem}{\rm
{In view of \cite[Theorem 2.3]{DLZ}, if $n = 2$, then any almost complex structure $J$ is
${\mathcal C}^{\infty}$ pure and full; therefore by Theorem \ref{harmonic} $J$ is
pure and full.}}
\end{rem}
\section{Hard Lefschetz condition} \label{hardLefschetz}
Let $(M, \omega)$ be a symplectic manifold of dimension $2n$. For any pair of (real or complex)
$k$-forms $\alpha,\beta$ on $M$, denote
by $\omega(\alpha,\beta)$ the bilinear form induced by the symplectic form $\omega$ on
the space $\Omega^k(M)$ of $k$-forms. Then if $\beta\in\Omega^{k}(M)$, for any
$\alpha\in\Omega^{k}(M)$, the following
$$
\alpha\wedge*_{\omega}\beta =\omega(\alpha,\beta)\frac{\omega^n}{n!}
$$
defines the {\em symplectic Hodge operator} $*_{\omega}:\Omega^{k}(M)\to\Omega^{2n-k}(M)$ on
$(M,\omega)$. Then,
according to the above definition of the symplectic Hodge operator, the {\em symplectic codifferential}
is the
operator $d^{*_\omega}:\Omega^k(M)\to\Omega^{k-1}(M)$ defined by
$$
d^{*_\omega}\alpha=(-1)^{k+1}*_{\omega} d *_{\omega}\alpha\,,
$$
for any $\alpha\in\Omega^k(M)$. By definition, a $k$-form $\alpha$ is said to be {\em symplectic harmonic}
if it satisfies
$$
d\alpha=d^{*_\omega}\alpha=0\,.
$$

By \cite[Remark 2.4.4]{Brylinski} on an almost-K\"ahler manifold $(M, \omega, J, g)$ we have
these two properties:
\begin{enumerate}
\item[I)] for any $(p, q)$-form $\alpha$, $*_{\omega} (\alpha) = i^{p - q} *_g (\alpha)$, where $*_{\omega}$
denotes the symplectic Hodge operator;\vskip.2truecm
\item[II)] if $\alpha$ is harmonic of type $(p,q)$, then also $*_{\omega} d*_{\omega} (\alpha) =0$,
and, consequently,
$\alpha$ is symplectic harmonic.
\end{enumerate}
Since $d(*_{\omega} \omega) =0$, the symplectic form $\omega$ on an almost-K\"ahler
manifold $(M, \omega, J, g)$ is harmonic.\vskip.5truecm
We recall that a $2n$-dimensional symplectic manifold $(M, \omega)$ satisfies the {\em Hard Lefschetz
condition} if the cup-product maps:
$$
[\omega]^k: H^{n - k} (M) \to H^{n + k} (M), \quad 0 \leq k \leq n
$$
are isomorphisms.\newline
By \cite{Mathieu}, a compact symplectic manifold $(M,\omega)$ satisfies the Hard Lefschetz condition
if and only if
any de Rham cohomology class has a symplectic harmonic representative.\newline
A classical result by \cite{DGM} states
that, if $(M, \omega, J)$ is a compact K\"ahler, i.e. if $J$ is an
integrable almost-K\"ahler structure on a compact manifold, then it satisfies the Hard Lefschetz condition.

A natural question is to see if there is any relation in the case
of an almost-K\"ahler structure $(\omega, J)$ between the
condition for $J$ to be pure and full and
the condition for $\omega$ to satisfy the Hard Lefschetz property.
We can prove the following
\begin{theorem} Let $(M, \omega)$ be a $2n$-dimensional compact symplectic manifold which satisfies the Hard
Lefschetz condition and $J$ be a ${\mathcal C}^{\infty}$ pure and
full almost complex structure calibrated by $\omega$. Then $J$ is
pure and full.
\end{theorem}
\begin{proof} If $n =2$ the result follows by Theorem \ref{harmonic}. If $n > 2$, we know already that $J$ is pure and we
have to show that
$$
H_{2} (M, \R) =H_{1,1}^J (M)_{\R} \oplus H_{(2,0),(0,2)}^J
(M)_{\R}\,.
$$
Let $a=[T]\in H_2(M,\R)$. Then we can represent $a$ as a de Rham class
$a=[\alpha]$, where $\alpha\in\Omega^{2n-2}(M)$ is $d$-closed. By the assumptions, $(M,\omega)$ satisfies
the Hard Lefschetz condition. Therefore, there exists $b\in H^2(M,\R)$, $b=[\beta]$ such that $a=b\cup [\omega]^{n-2}$, i.e.
$$
[\beta\wedge\omega^{n-2}]=[\alpha]\,.
$$
Since $J$ is ${\mathcal C}^{\infty}$ pure and
full, it follows that
$$
[\beta]=[\varphi]+[\psi]\,,
$$
where $[\varphi]\in H^{1,1}_J (M)_{\R}$, $[\psi]\in H^{(2,0),(0,2)}_J (M)_{\R}$ respectively. Then we obtain
\begin{equation}\label{lefschetz}
a=[T]=[\beta\wedge\omega^{n-2}]=[\varphi\wedge\omega^{n-2}]+[\psi\wedge\omega^{n-2}]\,.
\end{equation}
Since $\varphi$, $\psi$ are real $2$-forms of type $(1,1)$, $(2,0)+(0,2)$ respectively and $\omega^{n-2}$ is a real
form of type $(n-2,n-2)$, condition \eqref{lefschetz} implies that
$$
a=[T]=[R]+[S]\,,
$$
where $R\in H_{1,1}^J (M)_{\R}$ and $S\in H_{(2,0),(0,2)}^J(M)_{\R}\,.$
Hence, $J$ is pure and full.
\end{proof}
Since if $n=2$, then it follows that every almost complex structure is ${\mathcal C}^{\infty}$ pure and
full (see \cite[Theorem 2.3]{DLZ}, it would be interesting to find for $n > 2$ an example of
compact symplectic manifold $(M, \omega)$
which satisfies Hard Lefschetz condition and with a non pure and full almost complex structure calibrated by $\omega$.
\section{Integrable pure and full almost-complex structures}
An integrable almost complex structure is closed (see \cite{HL}),
but in general it is not necessarily (${\mathcal C}^{\infty}$) pure and full. By
\cite{LIZhang} if $J$ is an integrable almost complex structure
and the Fr\"olicher spectral sequence degenerates at $E_1$, then
$J$ is pure and full. Any complex surface has this property, so in
real dimension $4$ the interesting case to study is the one of non
integrable almost complex structures.

For compact complex parallelizable manifolds, i.e. compact quotients of complex Lie
groups by discrete Lie subgroups, we
can prove the following
\begin{theorem} \label{complexparal} If $(M= \Gamma \backslash G, J)$ is a complex parallelizable manifold
and for the de Rham cohomology we have the isomorphism $H^2 (M, \R) \cong H^2 (\mathfrak g)$, then
$J$ is ${\mathcal C}^{\infty}$ full and therefore it is pure. Moreover,
$$
H^2 (M, \R) \cong {\mathcal Z}^{1,1}_{inv} \oplus {\mathcal Z}^{(2,0) + (0,2)}_{inv} /
{\mathcal B}^{(2,0),(0,2)}_{inv}.
$$
\end{theorem}
\begin{proof}
Denote by $\Omega^{p,q}_{inv}(M)$ the space of left-invariant complex forms of type $(p,q)$.
Since $J$ is bi-invariant, we have at the level of left-invariant forms
$d (\Omega^{1,0}_{inv}(M)) \subseteq \Omega^{2,0}_{inv}(M))$ and therefore
$$
\begin{array}{l}
d (\Omega^{1,1}_{inv} (M))_{\R} \subseteq (\Omega^{2,1}_{inv} (M) + \Omega^{1,2}_{inv} (M) )_{\R}\,,\\[7pt]
d (\Omega^{2,0}_{inv} (M) + \Omega^{0,2}_{inv} (M))_{\R}
\subseteq (\Omega^{3,0}_{inv} (M) + \Omega^{0,3}_{inv} (M) )_{\R}\,.
\end{array}
$$
Therefore
$$H^2(M, \R) \cong \left({\mathcal Z}^{1,1}_{inv} \oplus {\mathcal Z}^{(2,0) + (0,2)}_{inv} \right) /
{\mathcal B}_{inv}\,,
$$
since ${\mathcal B}^{1,1}_{inv} =0$. This implies that $J$ is ${\mathcal C}^{\infty}$ full, since any $[\alpha]
\in H^2(M, \R)$ can be written as
$$
[\alpha] = [\alpha_1] + [\alpha_2]
$$
with $\alpha_1$ and $\alpha_2$ left-invariant and
respectively belonging to ${\mathcal Z}^{1,1}_{inv}$ and to $ {\mathcal Z}^{(2,0) + (0,2)}_{inv}$.
As a consequence of Theorem \ref{harmonic} we get that $J$ is pure.
\end{proof}
Therefore we have the following
\begin{cor} \label{complexparalnilm}
Let $(M,J)$ be a complex parallelizable nilmanifold. Then $J$ is ${\mathcal C}^{\infty}$ full and it is pure.
\end{cor}
A classification of compact complex parallelizable solvmanifolds of complex dimension $3, 4, 5$
was obtained by Nakamura in \cite{Nakamura}. Therefore, for instance, we can apply
Corollary \ref{complexparalnilm} to the
Iwasawa manifold endowed with the natural bi-invariant complex
structure. The other complex solvmanifold of real dimension $6$ is
the Nakamura manifold. In this case we can also show that it
admits pure and full almost complex structures, even if its de
Rham cohomology is not isomorphic to the Chevalley cohomology
of the corresponding solvable Lie algebra.
\begin{ex} {\bf (Nakamura manifold)}
{\rm Consider the solvable
Lie group $G$ with structure equations
$$
\left\{ \begin{array}{l} d e^1 =0\,,\\[5pt]
d e^2 = e^{12} - e^{45}\,,\\[5pt]
d e^3 = -e^{13} + e^{46}\,,\\[5pt]
d e^4 =0\,,\\[5pt]
d e^5 = e^{15} - e^{24}\,,\\[5pt]
d e^6 = e^{16} + e^{34}\,.
\end{array}
\right.
$$
The Lie group $G$ is isomorphic to $\C^3$ with product $*$, defined in terms of
the complex coordinates on $\C^3$ $(z_1=x_1 + i x_4, z_2 = x_2 + i x_5, z_3 = x_3 + i x_6)$ by
$$
^t (z_1, z_2, z_3) *\, ^t(w_1, w_2, w_3) =\, ^t(z_1 + w_1, e^{- w_1} z_2 + w_2, e^{w_1} z_3 + w_3),
$$
for any $^t(z_1, z_2, z_3),\, ^t(w_1, w_2, w_3) \in \C^3$. \newline
The {\em Nakamura manifold} is the compact quotient $X = \Gamma \backslash G$
of $G$ by a uniform discrete subgroup $\Gamma$.

By \cite[Corollary 4.2]{debaT} we have
$$
\begin{array}{lcl} H^2 (X, \R) &=& \R < [e^{14}], [e^{26} - e^{35}], [e^{23} - e^{56}], [\cos(2 x_4)
(e^{23} + e^{56}) - \sin(2x_4) (e^{26} + e^{35})]\,, \\[7pt]
&&[\sin(2 x_4)(e^{23} + e^{56}) - \cos(2x_4) (e^{26} + e^{35})]>\,,
\end{array}
$$
i.e. in this case the de Rham cohomology of $M$ is not isomorphic to $H^* ({\mathfrak g})$. The previous
representatives are all harmonic forms.
The complex solvmanifold $X$ admits a left-invariant almost complex structure $J$ defined by the
$(1,0)$-forms:
$$
\eta^1 = e^1 + i e^4\,, \quad \eta^2 = e^3 + i e^5\,, \quad \eta^3 = e^6 + i e^2
$$
calibrated by the symplectic form
$$
\omega = e^{14} + e^{35} + e^{62}.
$$
We have that the forms
$$
e^{14}\,,\qquad e^{26} - e^{35},
$$
$$
\cos(2 x_4) (e^{23} + e^{56}) - \sin(2x_4) (e^{26} + e^{35})\,,\qquad
\sin(2 x_4) (e^{23} + e^{56}) - \cos(2x_4) (e^{26} + e^{35})
$$
are all of type $(1,1)$ with respect to $J$ and $e^{23} - e^{56}$ is of type $(2,0)$ with
respect to $J$. Then, by applying Theorem \ref{harmonic} we get that $J$ is pure and full.

Moreover, $X$ admits the bi-invariant complex structure $\tilde J$ defined by the $(1,0)$-forms
$$
\tilde \eta^1 = e^1 + i e^4, \quad\tilde \eta^2 = e^2 + i e^5, \quad \tilde \eta^3 = e^3 + i e^6.
$$
By Theorem \ref{harmonic} $J$ is pure and full.
}
\end{ex}
\section{A family of pure and full almost complex structures} \label{familyLafontaine}
In this section we will provide a family of pure and full almost complex structures on a nilmanifold of real
dimension $4$ and on two $2$-step solvmanifolds of real dimension $4$ and $6$.

From now on we will use the convention that the almost complex structure $J$ acts on $1$-forms as $(J \alpha) (X) = - \alpha (J X)$, for any $1$-form $\alpha$ and vector field $X$.

\subsection{A $4$-dimensional solvmanifold $M^4$} \label{Sol3}
Consider the $3$-dimensional completely solvable Lie group $\mbox{\em Sol}_3 = \R \ltimes_{\phi} \R^2$,
where $\phi(t) $ is the one-parameter subgroup defined by $\phi(t) = {\mbox {diag}} (e^t, e^{-t})$. By
\cite{FG} a compact quotient $M^4$ of the Lie group $\mbox{\em Sol}_3 \times \R$ with structure equations
\begin{equation} \label{eqstructuresol3}
\left\{
\begin{array}{l}
d e^1 =0\,,\\[5pt]
d e^2 =0\,,\\[5pt]
d e^3 = -e^{13}\,,\\[5pt]
d e^4 =e^{14}\,,
\end{array}
\right.
\end{equation}
by a uniform discrete subgroup admits a symplectic structure which satisfies the Hard Lefschetz condition.
The symplectic structure is given by
$$
\omega = e^{12} + e^{34}.
$$
The compact symplectic manifold is cohomologically K\"ahler (in fact, it has the same minimal model as
${\mathbb T}^2 \times S^2$) and it does not carry complex structure. The $\omega$-calibrated almost
complex structure $J_0$ defined by the $(1,0)$-forms$$
\eta^1 = e^1 + i e^2, \, \eta^2 = e^3 + i e^4
$$
is pure and full, since by \cite{Hattori} we have that
$$
H^2(M^4, \R)= \R < [e^{12}], [e^{34}]>
$$
and both the forms $e^{12}$ and $e^{34}$ are of type $(1,1)$. We
can show that this compact symplectic manifold $(M^4, \omega)$
admits a family of $\omega$-calibrated almost complex structures
which are pure and full. Indeed, let
us construct a ${\mathcal C}^\infty$ pure and full deformation
corresponding to the matrix
$$
J_t = (I + L_t) J_0 (I + L_t)^{-1}
$$
with respect to the coframe $(e^1, \ldots,e^4)$, where $I$ is the identity matrix of order $4$,
$J_0$ is the matrix
$$
J_0 = \left(
\begin{array} {cccc}
0&-1&0&0\\
1&0&0&0\\
0&0&0&-1\\
0&0&1&0
\end{array} \right)
$$
and
$$
L_t = \left(
\begin{array} {cccc}
0&0&ta&0\\
0&0&0&-ta\\
ta&0&0&0\\
0&-ta&0&0
\end{array} \right)\,,
$$
with $t,a\in\R$ satisfying the condition
$$
4t^2 a^2 < 1.
$$
A direct computation gives
$$
J_t = \frac{1} {( t^2 a^2 - 1)} \left(
\begin{array} {cccc}
0&1+t^2a^2&0&2ta\\[10pt]
-1 - t^2 a^2 &0&2ta&0\\[10pt]
0&2ta&0&1 + t^2 a^2\\[10pt]
2ta&0&-1 - t^2 a^2&0
\end{array}
\right)\,.
$$
Therefore, by \cite{Audin} (see Section \ref{calibrated}) the almost
complex structure $J_t$ is $\omega$-calibrated.
A basis of $(1,0)$-forms for $J_t$ is given by
$$
\begin{array}{l}
\varphi^1_t = e^1 + i \left(-\frac{1 + t^2 a^2}{t^2a^2-1} e^2 +\frac{2ta}{t^2 a^2-1} e^4\right)\,,\\[10pt]
\varphi^2_t = e^3 + i \left(-\frac{2ta}{t^2 a^2 -1} e^2 -\frac{1 + t^2 a^2}{t^2a^2-1} e^4\right)\,.
\end{array}
$$
Thus, we get that
$J_t$ is a curve of pure and full almost complex structures on $M^4$, calibrated
by the symplectic form $\omega$.
\subsection{A $4$-dimensional nilmanifold $\tilde M^4$} \label{M2nil} The only two
nilmanifolds of dimension $4$ which admit an
almost-K\"ahler structure $(J, \omega)$ are the Kodaira-Thurston
manifold and the $3$-step nilmanifold $\tilde M^4 = \Gamma
\backslash N$, compact quotient of the $3$-step nilpotent Lie
group $N$, whose nilpotent Lie algebra ${\mathfrak {n}}$ has
structure equations
\begin{equation} \label{eqstructurenil3}
\left\{
\begin{array}{l}
d e^j =0\,, \qquad j = 1,2,\\[3pt]
d e^3 = e^{14}\,,\\[3pt]
d e^4 = e^{12}\,,
\end{array}
\right.
\end{equation}
where $e^{ij}$ stands for $e^i \wedge e^j$.

The Kodaira-Thurston manifold has a pure and full integrable
almost complex structure since in this case the Fr\"olicher spectral
sequence degenerates at $E_1$. Since $b_1(\tilde M^4) = 2$, $\tilde M^4$ does
not admit any integrable almost complex structure (see
\cite{FGG}).
\begin{prop} The $3$-step nilmanifold $\tilde M^4$ admits a pure and full
almost-K\"ahler almost complex structure and a
pure and full (non almost-K\"ahler) almost complex structure.
\end{prop}
\begin{proof} Consider the almost complex structure $J_0$ defined by the $(1,0)$-forms
$$
\eta^1 = e^1 + i e^3, \,\,\, \eta^2 = e^2 + i e^4.
$$
Then $J_0$ is almost-K\"ahler with respect to
$$
\omega = e^{13} + e^{24}.
$$
By Nomizu's Theorem \cite{No} one has that the de Rham cohomology of the nilmanifold $\tilde M^4$ is
isomorphic to the
Chevalley-Eilenberg cohomology of the Lie algebra ${\mathfrak n}$. In particular, we have
$$
H^1 (\tilde M^4, \R)= \R <[e^1], [e^2]>, \quad H^2(\tilde M^4, \R) = \R < [e^{13}], [e^{24}]>.
$$
As shown in \cite{Mathieu}, $(\tilde M^4, \omega)$ does not satisfies hard Lefschetz property since
$$
e^1 \wedge \omega = d(e^{23}), \quad e^2 \wedge \omega = d(e^{34}).
$$

Since $e^{13}$ and $e^{24}$ are forms of type $(1,1)$ with respect to $J_0$ or
by \cite[Theorem 2.3]{DLZ}, we have
that $J_0$ is ${\mathcal C}^{\infty}$ pure and full with
$$
H^{2} (\tilde M^4, \R)=H^{1,1}_{J_0} (\tilde M^4)_{\R}.
$$
Thus, by Theorem \ref{harmonic} $J_0$ is pure and full.

The nilmanifold $\tilde M^4$ admits also an almost complex structure
$\tilde J_0$ which is pure and full, but not almost-K\"ahler.
Consider on $\tilde M^4$ the left-invariant almost complex structure defined by the $(1,0)$-forms $$
\tilde \eta_1 = e^1 + i e^2\,,\quad \tilde \eta_2 = e^3 + i e^4\,.
$$
In this case we have
$$
H^{1,1}_{\tilde J_0} (\tilde M^4)_{\R} = \R< [e^{13} + e^{24}]>\,, \quad
H^{(2,0)+(0,2)}_{\tilde J_0} (\tilde M^4)_{\R}= \R <[e^{13} - e^{24}]>
$$
and, again by using Theorem \ref{harmonic} we have that $\tilde
J_0$ is full and pure.

In order to prove that $\tilde J_0$ is not almost-K\"ahler, we
start to show that there is no symplectic left-invariant $2$-form
$\tilde \omega$ on $\tilde M^4$ such that $\tilde J_0$ is calibrated by
$\tilde \omega$. Indeed, such a $2$-form has to be closed and of type
$(1,1)$ with respect to $\tilde J_0$, i.e. of the form
$$
\tilde \omega = a e^{12} + b (e^{13} + e^{24})\,,
$$
with $a, b \in \R$, but then $\tilde \omega (e_3, e_4) =0.$

By using the same argument as in \cite{FGr}, we then show that
there exists no symplectic $2$-form $\omega$ on $\tilde M^4$ such that
$\tilde J_0$ is $\omega$-calibrated. Indeed, if such a non
left-invariant symplectic $2$-form $\omega$ exists, we can
consider the left-invariant $2$-form $\tilde \omega$, defined by
$$
\tilde \omega (X_1, X_2) = \int_{x\in \tilde M^4} \omega_x (X_1 \vert_x,
X_2 \vert_x) d \mu,
$$
where $X_1, X_2$ are projections of left -invariant vector fields
from $G_2$ to $\tilde M^4$ and $d \mu$ is a bi-invariant volume form on
$\tilde M^4$. The form $\tilde \omega$ is then a closed $(1,1)$-form and
$\tilde J_0$ is calibrated by $\tilde \omega$, but this is a
contradiction.
\end{proof}

We can construct on $\tilde M^4$ a family of pure and full almost-K\"ahler
structures. In order to do this, we will
change a little
bit our notation. We change the basis of $1$-forms so that, in the new
coframe $\{f^1,\ldots ,f^4\}$, the
structure equations are
$$
\left\{
\begin{array}{l}
d f^1 =0\,, \\[5pt]
d f^2 = f^{14}\,,\\[5pt]
df^3=0\,,\\[5pt]
d f^4 = f^{13}\,.
\end{array}
\right.
$$
Then the almost-K\"ahler $(\omega,J_0)$ structure is given by
$$
\omega = f^{12} + f^{34}\,,\qquad \varphi^1=f^1+if^2\,,\quad\varphi^2=f^3+if^4\,.
$$
Hence
$$
H^2(\tilde M^4, \R)=\R <[f^{12}], [f^{34}]>=H^{1,1}_{J_0}(\tilde M^4)_{\R}\,.
$$
Let us construct a ${\mathcal C}^\infty$ pure and full deformation of $(\omega,J_0)$
depending on four real parameters
and corresponding to the matrix
$$
J_t = (I + L_t) J_0 (I + L_t)^{-1}
$$
with respect to the dual basis $(f^1, \ldots, f^4)$, where $I$ is the identity matrix of order $4$,
$J_0$ is the matrix
$$
J_0 = \left(
\begin{array} {cccc}
0&-1&0&0\\
1&0&0&0\\
0&0&0&-1\\
0&0&1&0
\end{array} \right)
$$
and
$$
L_t = \left(
\begin{array} {cccc}
ta&tb&0&0\\
tb&-ta&0&0\\
0&0&tp&tq\\
0&0&tq&-tp
\end{array} \right)\,,
$$
with $t,a,b,p,q\in\R$ and satisfying
$$
2t^2(a^2+b^2+p^2+q^2) < 1.
$$
A direct computation gives
$$
J_t = \left(
\begin{array} {cccc}
-\frac{2tb}{t^2(a^2+b^2)-1}&\frac{(ta+1)^2+t^2b^2}{t^2(a^2+b^2)-1}&0&0\\[10pt]
-\frac{(ta-1)^2+t^2b^2}{t^2(a^2+b^2)-1}&\frac{2tb}{t^2(a^2+b^2)-1}&0&0\\[10pt]
0&0&-\frac{2tq}{t^2(p^2+q^2)-1}&\frac{(tp+1)^2+t^2q^2}{t^2(p^2+q^2)-1}\\[10pt]
0&0&-\frac{(tp-1)^2+t^2q^2}{t^2(p^2+q^2)-1}&\frac{2tq}{t^2(p^2+q^2)-1}
\end{array}
\right)\,.
$$
Then, by \cite{Audin} (see Section \ref{calibrated}) the almost
complex structure $J_t$ is $\omega$-calibrated.
%Therefore a basis of $(1,0)$-forms for $J_t$ is given by
%$$
%\begin{array}{l}
%\varphi^1_t = f^1 + i \left(-\frac{2tb}{t^2(a^2+b^2)-1}f^1 -\frac{(ta-1)^2+t^2b^2}{t^2(a^2+b^2)-1}f^2\right)\,,\\[10pt]
%\varphi^2_t = f^3 + i \left(-\frac{2tq}{t^2(p^2+q^2)-1}f^3 -\frac{(tp-1)^2+t^2q^2}{t^2(p^2+q^2)-1}f^4\right)\,.
%\end{array}
%$$
Moreover, we obtain
$$
H^2(\tilde M^4, \R)=H^{1,1}_{J_t}(\tilde M^4)_\R\,,
$$
i.e. $J_t$ is a curve of $\mathcal{C}^\infty$ pure and full almost
complex structures on $\tilde M^4$, calibrated by the symplectic form
$\omega$. As a consequence of Theorem \ref{harmonic}, $J_t$ is a
curve of pure and full almost complex structures.

\smallskip

\subsection{A $6$-dimensional solvmanifold} Consider the completely
solvable Lie algebra $\mathfrak s$ with structure equations
$$
\left\{
\begin{array}{l}
d f^1 =0\,,\\[5pt]
d f^2 = - f^{12}\,,\\[5pt]
d f^3 = f^{34}\,,\\[5pt]
d f^4 = 0,\\[5pt]
d f^5 = f^{15}\,,\\[5pt]
d f^6 = f^{46}\,.
\end{array}
\right.
$$
The Lie algebra is a direct sum of two copies of the $3$-dimensional solvable Lie algebra $\mathfrak{sol}(3)$ and
by \cite[p. 20]{Auslander} the corresponding simply connected Lie group $S$ admits a compact quotient $M^6 = \Gamma \backslash S$.

By Hattori's Theorem \cite{Hattori} we have that the de Rham cohomology of the solvmanifold $M^6$
is isomorphic to the
Chevalley-Eilenberg cohomology of the Lie algebra ${\mathfrak s}$. In particular
$$
H^2(M^6, \R) = \R < [f^{14}], [f^{25}], [f^{36}]>.
$$
Consider the almost complex structure $J_0$ defined by the $(1,0)$-forms$$
\varphi^1 = f^1 + i f^4, \, \varphi^2 = f^2 + i f^5, \, \varphi^3 = f^3 + i f^6.
$$
Then $J_0$ is almost-K\"ahler with respect to
$$
\omega = f^{14} + f^{25}+ f^{36}
$$
and by \cite{FM} $(M^6, J_0, \omega)$ satisfies the Hard Lefschetz
condition. Moreover, $H^2(M^6, \R) = H^{1,1}_{J_0} (M^6)_{\R}$, since $f^{14}, f^{25}, f^{46}$
are of type $(1,1)$ with respect to $J_0$.

Define the family of almost complex structure corresponding to the matrix
$$
J_t = (I + L_t) J_0 (I + L_t)^{-1}
$$
with respect to the basis $(f^1, \ldots, f^6)$, where $I$ is the identity matrix of order $3$,
$J_0$ is the matrix
$$
J_0 = \left( \begin{array} {cc} 0&-I \\ I&0 \end{array} \right)
$$
and $L_t$ is the real symmetric matrix of order $6$ given by
$$
L_t = \left(\begin{array}{cc} 0&tI\\
tI &0 \end{array} \right)
$$
and such that $6t^2<1$.
Then, by \cite{Audin} (see Section \ref{calibrated}) the almost complex structure $J_t$ is
$\omega$-calibrated.

Consequently,
$$
J_t = \left( \begin{array}{cc} \frac{2t} {1 - t^2} I & -\frac{1+t^2}{1-t^2} I\\ [10pt]
\frac{1+t^2}{1-t^2} I & -\frac{2t} {1 - t^2} I
\end{array} \right)
$$
is a curve of almost-K\"ahler structures on $M$.\newline
Any almost complex structure $J_t$ is $\mathcal{C}^\infty$ pure, since it is $\omega$-calibrated. Moreover,
a basis of $(1,0)$-forms for $J_t$ is given by
$$
\begin{array}{l}
\varphi^1_t = f^1 + i \left(\frac{2t} {(1-t^2)} f^1 +\frac{1+t^2}{1-t^2}f^4 \right)\,,\\[10pt]
\varphi^2_t = f^2 + i \left( \frac{2t} {(1-t^2)}f^2 + \frac{1+t^2}{1-t^2}f^5 \right)\,,\\[10pt]
\varphi^3_t = f^3 + i \left( \frac{2t} {(1-t^2)} f^3 + \frac{1+t^2}{1-t^2}f^6 \right)\,.
\end{array}
$$
Then $J_t$ is a family of $\mathcal{C}^\infty$ pure and full almost complex structures, that
it is actually pure and full, since
$\varphi^1_t\wedge\overline{\varphi}^1_t$, $\varphi^2_t\wedge\overline{\varphi}^2_t$,
$\varphi^3_t\wedge\overline{\varphi}^3_t$
are harmonic (see Theorem \ref{harmonic}).

Consider the family $\tilde J_t$ of almost complex structures defined by
$$
\tilde J_t = J_0 (I + \tilde L_t) (I - \tilde L_t)^{-1},
$$
where $I$ is the identity matrix of order $6$ and $\tilde L_t$ is the non symmetric matrix given by
$$
\tilde L_t = \left( \begin{array}{cccccc} 0&0&0&0&0&0\\ 0&0&0&t&0&0\\ 0&0&0&0&0&0\\ 0&0&0&0&0&0\\
t&0&0&0&0&0\\ 0&0&0&0&0&0
\end{array} \right)
$$
Since $\tilde L_t J_0 = - J_0 \tilde L_t$, we have that $\tilde{J}_t$ is a family of $\omega$-tamed almost complex structures
defined by the $(1, 0)$-forms
$$
\begin{array}{l}
\tilde \varphi^1_t = f^1 + i (- 2t f^2 +f^4)\,,\\[10pt]
\tilde \varphi^2_t = f^2 + i f^5 \, , \\[10pt]
\tilde \varphi^3_t = f^3 + i f^6.
\end{array}
$$
Moreover, $\tilde{J}_t$
is a family of pure and full almost complex structures.

\section{The deformations $J_{\alpha}$}
In \cite{Lee} it was shown that on a almost-K\"ahler manifold $M$ of real dimension $2n$ there is natural
infinite-dimensional family of
almost complex structures $J_{\alpha}$ parametrized by the
$J$-anti invariant real $2$-forms on $M$.
Indeed, given a $2n$-dimensional almost-K\"ahler manifold $(M, J, \omega, g)$ one may consider
as in \cite{Lee} a natural infinite dimensional
family of almost complex structures parametrized by differential forms
in $\Omega^{(2,0),(0,2)}_J (M)_{\R}$. Each form $\alpha$
in $\Omega^{(2,0),(0,2)}_J (M)_{\R}$ defines an
endomorphism $K_{\alpha}$ of the tangent bundle $TM$ by
$$
g( X, K_{\alpha} Y) = \alpha (X,Y),
$$
for any pair of vector fields $X$ and $Y$.
Moreover, we have (see \cite{Lee})
$$
g(K_{\alpha} X, Y) = - g (X, K_{\alpha} Y), \quad J K_{\alpha} = - K_{\alpha} J,
\quad g (JX, K_{\alpha} X) =0,
$$
for any $X$ and $Y$. Since $J K_{\alpha}$ is skew-adjoint for each
$\alpha \in \Omega^{(2,0),(0,2)}_J(M)_{\R}$, we have that
$Id+J K_{\alpha}$ is invertible and hence
\begin{equation} \label{Jalpha}
J_{\alpha} = (Id+ J K_{\alpha} )^{-1} J (Id+J K_{\alpha} )
\end{equation}
is an almost complex structure on $M$. Moreover, by \cite[Proposition 1.5]{Lee} $J_{\alpha}$
satisfies the following condition:
\begin{equation} \label{compatibiltyJalpha}
g(J_{\alpha} X, J_{\alpha} Y) = g(X, Y).
\end{equation}
Moreover, if $n =2$, since in this case $K_ {\alpha}^2 = - \vert
\vert \alpha \vert \vert^2 Id$, then we have that $$J_{\alpha} = \frac{1
- \vert \vert \alpha \vert \vert^2} {1 + \vert \vert \alpha \vert
\vert^2} J - \frac {2} {1 + \vert \alpha \vert^2} K_{\alpha}$$
(see \cite[Proposition 1.5]{Lee}) and if we denote by
$J^*_{\alpha}$ the adjoint of $J_{\alpha}$ we have that the
symmetric part of $J^*_{\alpha} J$ is $\frac{1 - \vert \vert
\alpha \vert \vert^2} {1 + \vert \vert \alpha \vert \vert^2} Id$
and
$$
J^*_{\alpha} J J_{\alpha} = J + 4 \frac{ (\vert \vert \alpha \vert
\vert^2 - 1)}{(1 + \vert \vert \alpha \vert \vert^2)^2}
K_{\alpha}.
$$
Therefore
\begin{enumerate}
\item $J_{\alpha}$ is $\omega$-tamed if and only $\vert \vert \alpha \vert \vert^2 < 1$;\vskip.1truecm
\item $J_{\alpha}$ is $\omega$-calibrated if and only $\alpha =0$
or $\vert \vert \alpha \vert \vert^2 =1$.
\end{enumerate}
\smallskip

In this section we want to study the deformations $J_{\alpha}$ of a pure and full almost-K\"ahler
structure $J$.
Following \cite{Lee}, let us recall some general computations in dimension $4$
in order to apply these to the solvmanifold $M^4$ and to the nilmanifold $\tilde M^4$
constructed respectively in
section \ref{Sol3} and \ref{M2nil}. Both these examples have an almost-K\"ahler
structure $(J_0, \omega)$ with $J_0$ pure and full.

Let $M$ be a $4$-dimensional almost-K\"ahler manifold $(M, J, g, \omega)$. Locally, we
may assume that the almost complex structure $J$ is given with respect to an
orthonormal coframe $(e^1, \ldots, e^4 )$ by
$$
J = \left( \begin{array}{cccc} 0&-1&0&0\\1&0&0&0\\ 0&0&0&-1\\ 0&0&1&0 \end{array} \right).
$$
Let $\alpha$ be any real form in the space $\Omega^{(2,0),(0,2)} (M)_{\R}.$ Then the $2$-form
$$
\alpha = a (e^{13} - e^{24}) + b (e^{14} + e^{23})
$$
defines the endomorphism $K_{\alpha}
$ given with respect to the orthonormal coframe $(e^1, \ldots, e^4)$ by
$$
K_{\alpha} = \left( \begin{array} {cccc}
0&0&a&b\\
0&0&b&-a\\
-a&-b&0&0\\
-b&a&0&0 \end{array} \right)
$$
with $a, b \in \R$. Then we can consider on $M$ the almost complex
structure $J_{\alpha}$ given by \eqref{Jalpha}:
\begin{equation} \label{explicitJalpha}
J_{\alpha} = \frac{1} {(1 + \vert \vert \alpha \vert \vert^2)}
\left(
\begin{array}{cccc}
0& -1 + \vert \vert \alpha \vert \vert^2& -2a& -2b\\[2pt]
1-\vert \vert \alpha \vert \vert^2& 0&-2b& 2a\\[2pt]
2a& 2b&0& -1+ \vert \vert \alpha \vert \vert^2\\[2pt]
2 b& -2a& 1-\vert \vert \alpha \vert \vert^2&0
\end{array}
\right),
\end{equation}
where $\vert \vert \alpha \vert \vert^2 = a^2 + b^2$.
\newline
The associated $(1,0)$-forms are:
$$
\begin{array} {l}
\varphi^1 = e^1 + i \frac{1} {(1 + \vert \vert \alpha \vert \vert^2)}
\left( (1- \vert \vert \alpha \vert \vert^2)
e^2 +2 a e^3 +2 b e^4 \right),\\[10pt]
\varphi^2 = e^3 + i \frac{1} {(1 + \vert \vert \alpha \vert \vert^2)}
\left( -2 a e^1 +2 b e^2 + (1- \vert \vert \alpha \vert \vert^2)
e^4 \right).
\end{array}
$$
Therefore, we can prove the following
\begin{prop}
The almost-K\"ahler solvmanifold $(M^4, J_0, \omega)$ and the
almost-K\"ahler nilmanifold $(\tilde M^4, J_0, \omega)$
constructed in section \ref{Sol3} and \ref{M2nil} admit a family of pure and full almost complex structures $J_{\alpha}$
parametrized
by the left-invariant forms $a (e^{13} - e^{24})$, with $a \notin \{-1, 0, 1\}$, where $(e^i)$ is the basis
of left-invariant forms satisfying \eqref{eqstructuresol3} in the case of $M^4$ and
$$
\left \{ \begin{array} {l}
d e^j =0\,, \,\qquad j = 1, 3\,, \\[5pt]
d e^2= e^{14}\,,\\[5pt]
d e^4 = e^{13}
\end{array}
\right.
$$
in the case of $\tilde{M}^4$.
\end{prop}
\begin{proof} For both almost-K\"ahler manifolds $M^4$ and $\tilde M^4$ we have
$$
\omega = e^{12} + e^{34}
$$
and $J_0 e^1 = e^2, J_0 e^3 = e^4$.

Consider the almost complex structure $J_{\alpha}$, given by \eqref{explicitJalpha}
with $b=0$, $a \neq \pm 1, a \neq 0$, then it turns out that $J_{\alpha}$ is pure and full in both cases.

In the case of $M^4$, we have
$$
\begin{array}{l}
i \varphi^1 \wedge \overline \varphi^1 = \frac{1}{(1 + |a|^2)} (- 4 a e^{13}+ 2 (1 - |a|^2) e^{12})\\[10pt]
i\varphi^2 \wedge \overline \varphi^2= \frac{1}{(1 + |a|^2)} (-4 a e^{13} -2 (1- |a|^2) e^{34})
\end{array}
$$ and for the nilmanifold $\tilde M^4$
$$
\begin{array}{l}
i (\varphi^1 \wedge \overline \varphi^2 + \overline \varphi^1
\wedge \varphi^2) = \frac{2 (-1 + \vert \vert \alpha \vert
\vert^2) } {(1 + \vert \vert \alpha \vert \vert^2)}
( e^{23} - e^{14})\,,\\[10pt]
i (\varphi^1 \wedge \varphi^2 - \overline \varphi^1 \wedge
\overline \varphi^2) = \frac{2 (-1 + \vert\vert \alpha \vert
\vert^2) } {(1 + \vert \vert \alpha \vert \vert^2)}
( e^{23} + e^{14})\,.\\
\end{array}
$$
\end{proof}

\end{document}